\documentclass[11pt]{article}
\usepackage{amssymb}
\usepackage{amsthm}
\usepackage{amsmath,amsfonts,amssymb}
\usepackage{graphicx}
\usepackage{algpseudocode}
\usepackage{algorithm}
\usepackage{natbib}

\newtheorem{theorem}{Theorem}[section]

\newtheorem{proposition}[theorem]{Proposition}

\newtheorem{definition}[theorem]{Definition}
\newtheorem{remark}[theorem]{Remark}

\def\sqr#1#2{\vbox{\hrule height .#2pt
\hbox{\vrule width .#2pt height #1pt \kern #1pt
\vrule width .#2pt}\hrule height .#2pt }}

\def\begi{\begin{itemize}}
\def\endi{\end{itemize}}
\def\bega{\begin{array}}
\def\enda{\end{array}}

\def\forall{\hbox{for every~ }}

\def\ds{\displaystyle}
\def\R{\mathbb{R}}

\def\forall{\hbox{for all}~}

\def\C{{\mathcal C}}

\def\bel{\begin{equation}\label}
\def\eeq{\end{equation}}

\begin{document}

\title{Lagrangian-based Hydrodynamic Model:\\
Freeway Traffic Estimation}
\author{Ke Han$^{a}\thanks{Corresponding author, e-mail: kxh323@psu.edu}$
\qquad Tao Yao$^{b}\thanks{e-mail: tyy1@engr.psu.edu}$
\qquad Terry L. Friesz$^{b}\thanks{e-mail: tfriesz@psu.edu}$
 \\\\
%EndAName
$^{a}$\textit{Department of Mathematics}\\
\textit{Pennsylvania State University, University Park, PA 16802, USA}\\
$^{b}$\textit{Department of  Industrial and Manufacturing Engineering}\\
\textit{Pennsylvania State University, University Park, PA 16802, USA}}
\date{}
\maketitle

\begin{abstract}
This paper is concerned with highway traffic estimation using traffic sensing data, in a Lagrangian-based modeling framework. We consider the Lighthill-Whitham-Richards (LWR) model \citep{LW, Richards} in Lagrangian-coordinates, and provide rigorous mathematical results  
 regarding the equivalence of viscosity solutions to the Hamilton-Jacobi equations in Eulerian and Lagrangian coordinates. We derive closed-form solutions to the Lagrangian-based Hamilton-Jacobi equation using the Lax-Hopf formula \citep{HJDaganzo, ABSP}, and discuss issues of fusing traffic data of various types into the Lagrangian-based H-J equation. A numerical study of the {\it Mobile Century} field experiment \citep{MC} demonstrates the unique modeling features and insights provided by the Lagrangian-based approach. 
\end{abstract}

\section{Introduction}

\subsection{General background}

Highway traffic estimation and prediction are of pivotal importance for traffic operation and management. From an estimation perspective, it is desirable to have a substantial amount of information available. This has been made possible by emerging technology in fixed and mobile sensing and by the advancement of cyber-physical infrastructure, see \cite{MC} for a specific discussion. Depending on the characteristics of sensing devices and the underlying modeling framework, traffic sensing can be categorized as Eulerian sensing (such as using loop detector, virtual tripwire, and video camera) and Lagrangian sensing (such as using GPS and smart phone).

This paper is concerned with using the Lighthill-Whitham-Richards model  to estimate traffic states, subject to constraints imposed by sensing data that are large in quantity, and high in dimension. The LWR model is described by a {\it partial differential equation} (PDE), more specifically, by a {\it scalar conservation law} (SCL).  For general mathematical background on SCL, the reader is referred to \cite{Bbook}.  The LWR SCL is closely related to a Hamilton-Jacobi equation, via the Moskowitz function \citep{Moskowitz}, also known as the Newell curve \citep{Newell}. It is shown in \cite{Evans} that a weak entropy solution to the LWR SCL leads to the viscosity solution of the corresponding H-J equation; and vise versa.

\subsection{Viability theory and Lax-Hopf formula}
From a PDE perspective, additional information on the traffic stream provided by sensors imposes multiple initial/boundary and/or internal boundary conditions that may lead to non-existence of a well-defined solution, namely, entropy solution to the scalar conservation law; or viscosity solution to the H-J equation. Recently, \cite{ABSP} extended the viscosity solution of H-J equations to a solution class known as the Barron-Jensen/Frankowska solutions \citep{BJ, Frankowska}. The generalized solution, called viability episolution \citep{Aubin}, is lower semi-continuous, and satisfies the initial/boundary and/or internal boundary conditions in an inequality sense. Such a relaxation, while excluding some nice mathematical properties of the solution, does allow the PDE to incorporate multiple value conditions and produce solutions that are well-informed of traffic states transmitted by the sensors. Based on this idea, a sequence of seminal papers, \cite{CC1, CC2, Convex} were created, which performed highway traffic estimation as well as data assimilation and inverse modeling.

\subsection{The LWR model in Lagrangian coordinate}

The LWR model is commonly formulated as a scalar conservation law in {\it Eulerian coordinate} (EC), i.e. with independent variables $t$ (time) and $x$ (location). A list of selected references in this line of research includes \cite{BH, BH1, CC1, CC2, HJDaganzo, HJDaganzo2, LWRDUE, Newell}. Another way of formulating the traffic dynamic is through {\it Lagrangian coordinate} (LC), which identifies $t$ (time) and $n$ (vehicle label). In contrast to EC, LC describes time evolution of traffic quantities associated with a moving vehicle.  Such a framework was initially introduced by \cite{CF} in the context of gas dynamics. LC was introduced to traffic flow modeling by \cite{HJDaganzo2}, and subsequently studied by \cite{Leclercq} and \cite{Yuan}. A more detailed review of the LWR model in both EC and LC is presented later in Section \ref{LWREL}.

The LWR PDE in LC can be derived from the Eulerian-based PDE; and vise versa. This procedure was given an intuitive illustration in \cite{HJDaganzo2}, provided that all quantities of interest were continuously differentiable. In \cite{Leclercq}, the authors related the equivalence of EC- and LC-based PDEs to an earlier work \citep{Wagner} on gas dynamics. However, we point out that \cite{Wagner} studied an Euler equation, which is a system of conservation laws; results regarding such a system do not apply immediately to the scalar conservation law of interest. In this paper, we provide a rigorous proof of the equivalence between the two systems, in the context of Hamilton-Jacobi equation and viscosity solutions. The result will apply to the LWR conservation laws even if the solutions are not continuous, let alone differentiable. Technical result is presented in Section \ref{LCequivalence}.

As mentioned before,  solution to the H-J equation has been extended to treat multiple value conditions, using viability theory \citep{Aubin}. In this paper, we apply the framework of viability theory to the Lagrangian-based Hamilton-Jacobi equations. The goal is to explore the unique modeling and computational advantages of the LC in highway traffic estimation. In particular, using the Lax-Hopf formula, we derive explicit solutions to the LC-based H-J equation with various value conditions including initial/intermediate condition, upstream/downstream boundary condition, internal boundary condition, and a combination of the above. This is presented in Section \ref{simplesol}. As pointed out by recent studies \cite{Leclercq} and \cite{Yuan}, LC-based conservation law has certain computational advantage over EC-based  ones, in the context of Godunov scheme \citep{Godunov}. Namely, the Godunov scheme reduces to an upwind scheme for LC-based conservation law, since the corresponding fundamental diagram is monotonically increasing. As we show in this paper, when solved with the Lax-Hopf formula, the LC-based H-J equation also reduces the computational complexity, compared with the EC-based H-J equation. Finally, we test the framework of Lagrangian-based traffic sensing and estimation with dataset obtained from the Mobile Century field experiment \citep{MC}. The Lagrangian-based approach provides traffic information associated with each moving vehicle, such as trajectory and speed. Such information are relatively difficult to obtain from the Eulerian-based approach: one has to perform partial inversion of the Moskowitz function.

\subsection{Organization}

The rest of this paper is organized as follows. In Section \ref{LWREL}, we briefly review the Lighthill-Whitham-Richards model in both Eulerian and Lagrangian coordinates.  A rigorous equivalence result regarding the EC-based and LC-based H-J equations is presented in Section \ref{LCequivalence}.  Section \ref{BJF} introduces the notion of viability episolution and the generalized Lax-Hopf formula. Data of different sources are discussed and integrated into the Lagrangian PDE. In section \ref{simplesol}, we derive explicit solutions to the Lagrangian-based H-J equation, using the Lax-Hopf formula.  In Section \ref{numerical}, we apply the Lagrangian-based computational paradigm to estimate real-world highway traffic stream.

\section{The LWR model in transformed coordinates}\label{LWREL}
In this section, we review the Lighthill-Whitham-Richards model in both Eulerian and Lagrangian coordinates. The transformation between these two coordinate systems is made through a function inversion under minor assumption.  In section \ref{LCequivalence}, rigorous result on the equivalence of the two systems is provided, using Hamilton-Jacobi equation and the notion of viscosity solution.

\subsection{The LWR model in Eulerian coordinates}\label{LCEulerian}
The classical LWR PDE reads
\bel{LWRE}
 \rho_t(t,\,x)+ f_x\big(\rho(t,\,x)\big)~=~0\qquad (t,\,x)\in[0,\,+\infty)\times[0,\,L]
\eeq
where the subscripts denote partial derivative. The scalar conservation law above describes the temporal-spatial evolution of average vehicle density $\rho(t,\,x)$ and average vehicle flow $f\big(\rho(t,\,x)\big)$. The density-flow relation is articulated by the {\it fundamental diagram} $f(\cdot)$
\bel{flow}
f\big(\rho\big)~=~\rho\,v(\rho)\qquad \rho\in[0,\,\rho_{max}]
\eeq
where $\rho_{max}$ is jam density. The average vehicle speed $v(\rho)\in[0,\,v_{max}]$ is a decreasing function of density; $v_{max}$ denotes the free flow speed. The fundamental diagram is assumed to be concave with maximum $M$ attained at $\rho^*$, where $M$ is the flow capacity.

We introduce the Moskowitz function $N(\cdot\,,\,\cdot)$, also know as the Newell-curve, such that 
\bel{identityxt}
N_t(t,\,x)~=~f\big(\rho(t,\,x)\big),\qquad N_x(t,\,x)~=~-\rho(t,\,x)\qquad \hbox{a.e.}
\eeq
Hereafter, ``a.e." stands for ``almost everywhere". $N(t,\,x)$ measures the cumulative number of vehicles that have passed location $x$ by time $t$. It has been shown in a number of circumstances that if $\rho(t,\,x)$ is the weak entropy solution to (\ref{LWRE}), then the corresponding Moskowitz function is the  viscosity solution to the following Hamilton-Jacobi equation
\bel{HJE}
N_t(t,\,x)-f\big(- N_x(t,\,x)\big)~=~0
\eeq

Note that viscosity solution to (\ref{HJE}) must be Lipschitz continuous, but not necessarily continuously differentiable due to the presence of shock waves in  $\rho(t,\,x)$. There exists, however, more general solution classes such as the viability episolution \citep{ABSP}, which is lower semi-continuous.

\subsection{The LWR model in Lagrangian coordinates }\label{LCLagrangian}

In Lagrangian coordinate system, a free variable $n$ is used to identify a particular vehicle. Note that in a continuum model, $n$ is treated as a real number.  The coordinate transformation from $(t,\,x)$ to $(t,\,n)$ is made by inverting the following  Moskowitz function 
\bel{N}
n~=~N(t,\,x)
\eeq
 Throughout  this article, we assume that the vehicle density is uniformly positive, i.e. there exists $\delta>0$ such that
\begin{equation}\label{reveqn1}
\rho(t,\,x)~\geq~\delta\qquad\forall~(t,\,x)\in[0,\,+\infty)\times[0,\,L]
\end{equation}
This assumption is not restrictive given the argument that  if the density vanishes in certain road segments, then the domain of study can be separated into several subdomains, with each one satisfying (\ref{reveqn1}).  

For each $t$ fixed,  (\ref{reveqn1}) implies that $N(t,\,\cdot)$ is a strictly decreasing function of  $x$, whose inverse is denoted by $X(t,\,\cdot)$. We have
\bel{X}
x~=~X(t,\,n)
\eeq
where $X(t,\,n)$ represents the location of vehicle labeled $n$ at time $t$. Identities (\ref{N}) and (\ref{X}) define the coordinate transformation. We have the following
\bel{identitynt}
X_t(t,\,n)~=~v(t,\,n),\qquad X_n(t,\,n)~=~-s(t,\,n)\qquad \hbox{a.e.}
\eeq
where $v(t,\,n)$ and $s(t,\,n)$ denote the average speed and spacing around vehicle $n$ at time $t$, respectively. Introducing the spacing-velocity relationship 
\bel{psidef}
v~=~\psi(s)~\doteq~f\left({1\over s}\right)\cdot s\in[0,\,v_{max}]\qquad s\in[1/\rho_{max},\, 1/\delta]
\eeq
we present the Hamilton-Jacobi equation in Lagrangian coordinate
\bel{HJnt}
X_t(t,\,n)-\psi\big(- X_n(t,\,n)\big)~=~0
\eeq

\subsection{Equivalence between (\ref{HJE}) and (\ref{HJnt})}\label{LCequivalence}
In this section, we provide rigorous mathematical analysis on the relationship between the two H-J equations in EC and in LC.  Our strategy is to show that if $N(t,\,x)$ is the viscosity solution to (\ref{HJE}), then its partial inverse as defined by (\ref{N}) and (\ref{X}), is the viscosity to (\ref{HJnt}). The converse will hold similarly. We begin with the definition of viscosity solution to Hamilton-Jacobi equation of the form
\bel{HJgeneral}
u_t+H(\nabla u)~=~0
\eeq
where the unknown $u(t,\,x)\in\R^m$; $\nabla u$ is the gradient of u with respect to $x$. In what follows, $C$, $C^1$ denotes the set of continuous and continuously differentiable functions, respectively.
\begin{definition}\label{visdef}
A function $u\in C(\Omega)$ is a  viscosity subsolution of (\ref{HJgeneral}) if, for every  $C^1$ function $\varphi=\varphi(t,\,x)$ such that $u-\varphi$ has a local maximum at $(t,\,x)$, there holds
\bel{subsol}
\varphi_t(t,\,x)+H(\nabla \varphi)~\leq~0
\eeq
Similarly, $u\in C(\Omega)$ is a viscosity supersolution of (\ref{HJgeneral}) if, for every $C^1$ function $\varphi=\varphi(t,\,x)$ such that $u-\varphi$ has a local minimum at $(t,\,x)$, there holds
\bel{supsol}
\varphi_t(t,\,x)+H(\nabla \varphi)~\geq~0
\eeq
We say that $u$ is a viscosity solution of (\ref{HJgeneral}) if it is both a supersolution and a subsolution in the viscosity sense.
\end{definition}

\begin{remark}
  If $u$ is a $C^1(\Omega)$ function and satisfies (\ref{HJgeneral}) at every $x\in\Omega$, then $u$ is also a solution in the viscosity sense. Conversely, if $u$ is a viscosity solution, then the equality must hold at every point $x$ where $u$ is differentiable. In particular, if $u$ is Lipschitz continuous, then it is almost everywhere differentiable, hence (\ref{HJgeneral}) holds almost everywhere in $\Omega$. 
\end{remark}

The next theorem  establishes the equivalence between viscosity solutions to (\ref{eqn1}) and (\ref{eqn2}), where $f(\cdot)$ and $\psi(\cdot)$ satisfy (\ref{psidef}).
\bel{eqn1}
N_t(t,\,x)-f\Big(-N_x(t,\,x)\Big)~=~0
\eeq
\bel{eqn2}
X_t(t,\,n)-\psi\Big(-X_n(t,\,n)\Big)~=~0
\eeq

\begin{theorem}\label{viscositythm}
Assume that $N(t,\,x), \, (t,\,x)\in\Omega\subset (-\infty,\,+\infty)\times\R^n$, is a viscosity solution to (\ref{eqn1}), furthermore, assume that the density is uniformly positive, i.e. $\rho(t,\,x)\geq \delta>0,\,\forall (t,\,x)\in\Omega$. Then function $X(t,\,\cdot)$ obtained by inverting $N(t,\,\cdot)$ is a viscosity solution to (\ref{eqn2}).
\end{theorem}

\begin{proof}
By assumption, $N(t,\,\cdot)$ is strictly decreasing with 
$$
\delta\,|x_1-x_2|~\leq~|N(t,\,x_1)-N(t,\,x_2)|~\leq~\rho_{max}\,|x_1-x_2|\qquad\forall ~x_1,\,x_2
$$
then $X(t,\,\cdot)$ is also strictly decreasing with
\bel{twoside}
1/\rho_{max}|n_1-n_2|~\leq~|X(t,\,n_1)-X(t,\,n_2)|~\leq~1/\delta\,|n_1-n_2|\qquad \forall~n_1,\,n_2
\eeq
We start by showing that $X(\cdot,\,\cdot)$ is a subsolution.  Indeed, given any $C^1$ function $Y=Y(t,\,x)$ such that $X-Y$ has a local maximum at $(t_0,\,n_0)$, without loss of generality, we assume $X(t_0,\,n_0)-Y(t_0,\,n_0)=0$, take the plane $t=t_0$ (see Figure \ref{mathproof}).

\begin{figure}[h!]
\centering
\includegraphics[width=0.45\textwidth]{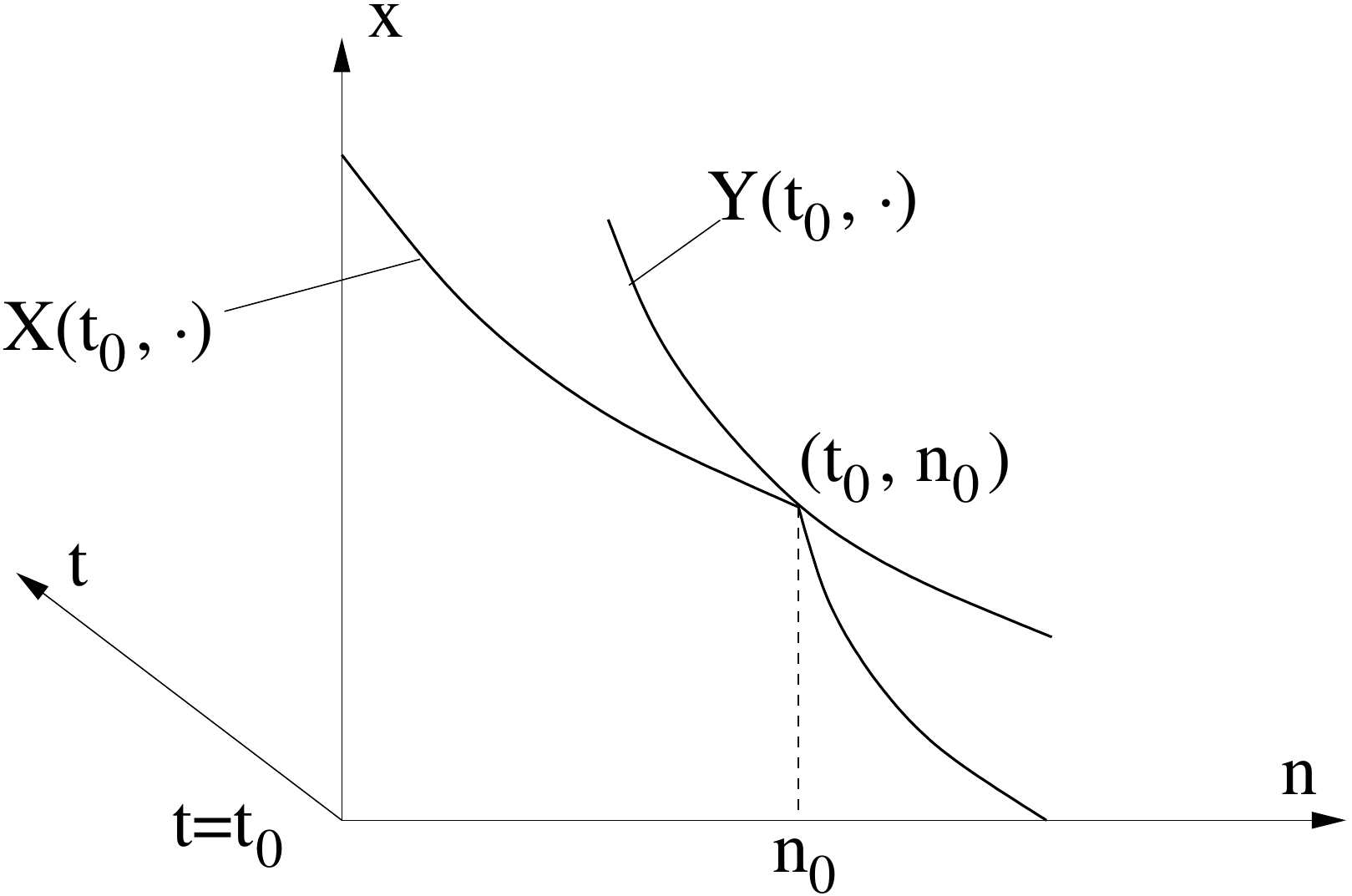}

\caption{Graphs of $X(t_0,\,\cdot)$ and $Y(t_0,\,\cdot)$}
\label{mathproof}
\end{figure}

\noindent Since $X-Y$ attains a local maximum at $(t_0,\,n_0)$, by (\ref{twoside}) there must hold $Y_n(t_0,\,n_0) <0$. By continuity, there exists a neighborhood $\Omega_1$ of $(t_0,\,n_0)$ such that $ Y_n(t,\,n)<0,~~\forall (t,\,n)\in\Omega_1$. Then, we may define $M(t,\,\cdot)$ to be the inverse of $Y(t,\,\cdot)$ in $\Omega_1$. Obviously, $M(t,\,n)\in C^1(\Omega_1)$, and $N-M$ attains a local maximum at $\big(t_0,\,X(t_0,\,n_0)\big)$. Using the fact that $N(t,\,x)$ is a viscosity solution and applying (\ref{subsol}), we deduce
\bel{a}
M_t(t,\,x)~\leq~f\Big(-M_x(t,\,x)\Big)
\eeq
Differentiating with respect to $t$ the identity $Y\big(t,\, M(t,\,x)\big)=x$,  and using (\ref{a}), we have
\bel{b}
0~=~Y_t + Y_n\,M_t~\geq~Y_t+Y_n\,f\Big(-M_x\Big)~=~Y_t+Y_n\,f\Big(-{1\over Y_n}\Big)
~=~Y_t-\psi(-Y_n)
\eeq
In the above deduction, we have used the fact that $M(t,\,\cdot)$ and $Y(t,\,\cdot)$ are  both $C^1$, and are inverse of each other. Therefore, differentiating  $n=M\big(t,\,Y(t,\,n)\big)$ w.r.t. $n$ yields $1=M_x\cdot Y_n$.

Since $Y$ is arbitrary, (\ref{b}) implies that $X(t,\,n)$ is a subsolution. The case for supersolution is completely similar.
\end{proof}

\begin{remark}
 Similar proof can be used to show the converse: given a viscosity solution $X(\cdot,\,\cdot)$ to (\ref{eqn2}),  $N(t,\,\cdot)$ obtained by inverting $X(t,\,\cdot)$ is a viscosity solution to (\ref{eqn1}).
\end{remark}

\section{Numerical algorithm and value conditions}\label{BJF}

In this section, we focus on the numerical scheme for the Hamilton-Jacobi equation
\bel{HJeqn}
\partial_t X(t,\,n)-\psi\big(-\partial_n X(t,\,n)\big)~=~0
\eeq
in the presence of  initial, boundary and internal boundary conditions. In order to avoid the issue of non-existence of solution, we adapt the notion of viability episolution \citep{ABSP}, and proceed with a variational method known as the Lax-Hopf formula \citep{ABSP, CC1, CC2}.

\subsection{Viability episolution to the Hamilton-Jacobi equation (\ref{HJeqn})}\label{LCviability}
 We consider a domain $[0,\,T]\times[N_1,\,N_2]$ for equation (\ref{HJeqn}), where  $T>0$ and $N_2> N_1> 0$. In the following definition, we make precise what we mean by value conditions, concerning (\ref{HJeqn}).
\begin{definition} 
A value condition $\C(\cdot,\,\cdot): \Omega\subset [0,\,T]\times[N_1,\,N_2]\rightarrow \mathbb{R}$ is a lower-semicontinuous function.
\end{definition}

\noindent One may extend a value condition $\mathcal{C}(\cdot,\,\cdot)$ to the entire domain by assigning $\mathcal{C}(t,\,n)=+\infty$ whenever $(t,\,n)\notin \Omega$. Such a  convention allows one to compare and manipulate value conditions with different domains.  In order to articulate the Lax-Hopf formula, we introduce the concave transformation $\psi^*(\cdot)$ of the Hamiltonian $\psi(\cdot)$
$$
\psi^*(p)~\doteq~\sup_{s\in[1/\rho_{max},\,+\infty)}\Big\{\psi(s)-p\,s\Big\}
$$ 
The following  Lax-Hopf formula provides semi-analytical representation of viability solution, given a value condition. 

\begin{theorem}
The viability episolution  to (\ref{HJeqn}) associated with value condition $\mathcal{C}(\cdot,\,\cdot)$ is given by
\bel{Lax}
X_{\mathcal{C}}(t,\,n)~=~\inf_{(u,\,T)\in\hbox{Dom}(\psi^*)\times\R_+}\big(\mathcal{C}(t-T,\,x+T\,u)+T\,\psi^*(u)\big)
\eeq
\end{theorem}
\begin{proof}
See \cite{ABSP}
\end{proof}

We indicate, by a subscript, dependence of the viability solution on its value condition. So that the solution to (\ref{HJeqn}) reads $X_{\mathcal{C}}(t,\,n)$.  As a consequence of formula (\ref{Lax}) the following property of {\it inf-morphism} holds.
\begin{proposition}\label{infmorphism}{\bf (inf-morphism property)}
Let $\mathcal{C}(\cdot,\,\cdot)$ be the point wise minimum of finitely many value conditions, 
$$
\mathcal{C}(t,\,n)~\doteq~\min_{i=1,\ldots,m}\mathcal{C}_i(t,\,n)\qquad \forall~~ (t,\,n)\in[0,\,T]\times[N_1,\,N_2]
$$
Then
\bel{infmor}
X_{\mathcal{C}}(t,\,n)~=~\min_{i=1,\ldots, m}X_{\mathcal{C}_i}(t,\,n)
\eeq
\end{proposition}

The inf-morphism property enables the H-J equation to incorporate one or more value conditions; it can also decompose a complex problem involving multiple value conditions into smaller subproblems, each one with a single condition.

\subsection{Value conditions for continuous solutions}\label{conditions}
The value conditions described in the previous section can be viewed as the mathematical abstraction of real-world measurements, obtained from Eulerian and Lagrangian sensing.  \cite{CC1,CC2}  have utilized the H-J equation in Eulerian coordinates to incorporate sensing data. In the rest of this paper, we not only extend their framework to Lagrangian coordinates, but also establish sufficient condition that guarantees the equivalence between these two systems, in the presence of a single value condition.

Consider again the two H-J equations (\ref{eqn1}) and (\ref{eqn2}) in EC and LC.  We make two assumptions on the value condition $\mathcal{C}$, as follows

{\bf (A1)} The domain of $\mathcal{C}$ is a continuous curve parametrized by $\tau\in[\tau_{min},\,\tau_{max}]$:
\begin{align*}
\hbox{Dom}(\mathcal{C})\subset [0,\,T]\times[N_1,\,N_2]\qquad &\Big(\hbox{respectively}~~  [0,\,T]\times[X_1,\,X_2]\Big)
\\
\hbox{Dom}(\mathcal{C})~=~\big(t(\tau),\,n(\tau)\big)\qquad &\Big(\hbox{respectively} ~~\big(t(\tau),\, x(\tau)\big)\Big)
\\
&\tau\in[\tau_{min},\,\tau_{max}]
\end{align*}

{\bf (A2)}  $\mathcal{C}\big(t(\cdot),\,n(\cdot)\big)$  is  a continuous function on $[\tau_{min},\,\tau_{max}]$. 

Given an arbitrary continuous value condition $\mathcal{C}(\cdot,\,\cdot)$ in the Eulerian domain (e.g. measurement of a loop detector), we need to determine how to fuse such datum into the Lagrangian based PDE. The next proposition provides an answer.

\begin{proposition}\label{sufficient} {\bf (Sufficient condition for equivalence of  value conditions)}

Let $\mathcal{C}^E\big(t(\tau),\,x(\tau)\big)$ and $\mathcal{C}^L(t(\tau),\,n(\tau)),\,\tau\in[\tau_{min},\,\tau_{max}]$ be two value conditions for (\ref{eqn1}) and (\ref{eqn2}), respectively. Then the solutions to (\ref{eqn1}) and (\ref{eqn2}) satisfying each value condition in the equality sense are equivalent if
$$
n(\tau)~=~\mathcal{C}^E\big(t(\tau),\,x(\tau)\big),\qquad x(\tau)~=~\mathcal{C}^L\big(t(\tau),\,n(\tau)\big)
$$
\end{proposition}
\begin{proof}
Let $N_{\mathcal{C}^E}(t,\,x)$ be the solution to (\ref{eqn1}) satisfying condition $\mathcal{C}^E$, let $\,X_{\mathcal{C}^L}(t,\,n)$ be the solution to (\ref{eqn2}) satisfying condition $\mathcal{C}^L$. For each $t$, since $N_{\mathcal{C}^E}(t,\,\cdot)$ is strictly increasing, we denote its inverse by $N^{-1}_{\mathcal{C}^E}(t,\,\cdot)$. Then by Theorem \ref{viscositythm},  $N^{-1}_{\mathcal{C}^E}(\cdot,\,\cdot)$ is a valid solution to the HJ equation (\ref{eqn2}).

On the other hand, for every $\tau\in[\tau_{min},\,\tau_{max}]$
\bel{123}
N^{-1}_{\mathcal{C}^E}\big(t(\tau),\,n(\tau)\big)~=~N^{-1}_{\mathcal{C}^E}\big(t(\tau),\,\mathcal{C}^E(t(\tau),\,x(\tau))\big)~=~N^{-1}_{\mathcal{C}^E}\Big(t(\tau),\, N_{\mathcal{C}^E}(t(\tau),\,x(\tau))\Big)~=~x(\tau)
\eeq
(\ref{123}) implies that $N^{-1}_{\mathcal{C}^E}(\cdot,\,\cdot)$ satisfies value condition $C^L\big(t(\tau),\,n(\tau)\big)$ and thus is the unique solution to (\ref{eqn2}) associated with value condition $\mathcal{C}^L$.   We conclude $N^{-1}_{\mathcal{C}^E}(t,\,n)=X_{\mathcal{C}^L}(t,\,n)$.
\end{proof}

\noindent Proposition \ref{sufficient} suggests that, by simply invert the Eulerian (Lagrangian) value condition $\mathcal{C}(t(\tau),\,\cdot)$, one can obtain its dual version in LC (EC). The dual value condition yields solution that is equivalent to the original solution.

\section{Explicit solution with piecewise affine value conditions}\label{simplesol}
In this section we apply the Lax-Hopf formula (\ref{Lax}) to for an explicit solution in the Lagrangian coordinates. To that end, we assume that the value conditions are {\it piecewise affine} (PWA). This assumption is justified on the ground of (1) an easily constructible explicit solution representation, with PWA conditions; and (2) standard result on PWA approximation of any piecewise continuous functions or curves.

\subsection{Piecewise affine value conditions}\label{PWAdef}

We start with articulating piecewise affine value conditions, which includes initial/intermediate, upstream, downstream and internal conditions. 
\begin{definition}\label{initial}
{\bf (PWA initial/intermediate condition)}

Set $t=t_0\geq 0$, given real numbers $s_i\geq0,\,n_i,\, i\in\{1,\ldots, m_{ini}\}$, the $j^{th}$ affine component of initial/intermediate condition is 
\bel{pwaini}
\mathcal{C}_{ini}(t,\,n)~=~-s_j\,n+d_j,\qquad (t,\,n)\in\{t_0\}\times[n_j,\,n_{j+1}]
\eeq
To ensure continuity, we require
$$
d_j~=~s_j\,n_j-\sum_{l=1}^{j-1}(n_{l+1}-n_l),\qquad j~=~2,\ldots, m_{ini}
$$
\end{definition}

\begin{definition}\label{upstream}
{\bf (PWA upstream boundary condition)}

Fix $n=N_1$, given real numbers $v^i\geq 0,\, t^i,\,i\in\{1,\,\ldots,\,m_{up}\}$, the $j^{th}$ affine component of upstream boundary condition is defined as
\bel{pwaup}
\mathcal{C}_{up}^j(t,\,n)~=~v^j\,t+b^j,\qquad (t,\,n)\in[t^j,\,t^{j+1}]\times\{N_1\}
\eeq
To ensure continuity of upstream boundary condition, we set 
$$
b^j~=~-v^j\,t^j+\sum_{l=1}^{j-1}(t^{l+1}-t^l)v^l,\qquad j~=~2,\ldots, m_{up}
$$
\end{definition}

\begin{definition}\label{downstream}
{\bf (PWA downstream boundary condition)}

Fix $n=N_2$, given real numbers $v_i\geq 0$, $t_i,\, i\in\{1,\ldots, m_{down}\}$, the $j^{th}$ affine component of downstream boundary condition is defined as
\bel{pwadown}
\mathcal{C}_{down}^j(t,\,n)~=~v_j\,t+b_j,\qquad (t,\,n)\in[t_j,\,t_{j+1}]\times\{N_2\}
\eeq
where
$$
b_j~=~-v_j\,t_j+\sum_{l=1}^{j-1}(t_{l+1}-t_l)v_l,\qquad j~=~1,\ldots, m_{down}
$$
\end{definition}

\begin{definition}\label{internal}
{\bf  (Affine internal boundary condition)}

Given real numbers $\alpha,\,\beta, t_{min},\,t_{max},\,n_{min},\,n_{max}, \, r\geq 0$, the affine internal boundary condition is defined as
\bel{affinternal}
\mathcal{C}_{int}(t,\,n)~=~\beta+\alpha\,(t-t_{min})\qquad t\in[t_{min},\,t_{max}],~~n=n_{min}+r(t-t_{min})
\eeq

\end{definition}

Recall that the domain of our consideration is $[0,\,T]\times[N_1,\,N_2]$, thus the upstream/downstream boundary conditions refer to (part of) the trajectories of the first and last car within our scope.

\subsection{Explicit formulae for viability episolutions}\label{LCexplicit}

In the presence of piecewise affine (PWA) data, the solution of Lagrangian equation (\ref{eqn2}) with a piecewise affine Hamiltonian can be computed explicitly, as given by the Lax-Hopf formula (\ref{Lax}).  We  denote $s_{min}\doteq 1/\rho_{max},\,s^*\doteq 1/\rho^*$, let $k>0$,  define
\bel{pwapsi}
\psi(s)~=~\begin{cases} k\,s,\qquad& s\in[s_{min},\, s^*]\\ v_{max},\qquad & s\in[s^*,\, +\infty)\end{cases}
\eeq
Notice that this fundamental diagram corresponds to a triangular density-flow relationship. Moreover, the concave transformation of $\psi(\cdot)$ reads
\bel{pwapsis}
\psi^*(u)~=~s^*\,(k-u)\qquad u\in[0,\,k]
\eeq

\begin{proposition}\label{explicit2}
Given each affine value condition defined in (\ref{pwaini})-(\ref{affinternal}), and Hamiltonian (\ref{pwapsi}), the solution to the Lagrangian Hamilton-Jacobi equation (\ref{eqn2}) are respectively

\begin{itemize}

\item[1.] Initial/intermediate value problem

if $s_j~\leq s^*$, 
\bel{inisimple1}
X_{ini}^j(t,\,n)=\begin{cases}
-s_j\,n+d_j+(t-t_0)k\,s_j,\\
n_j+k(t-t_0)\leq n\leq n_{j+1}+k(t-t_0);\\
\\
\ds-s_j\,n_j+d_j+(t-t_0)\,s^*\,k-s^*\,(n-n_j),\\
0\leq n-n_j\leq k(t-t_0).
\end{cases}
\eeq

if $s_j~>s^*$, 
\bel{inisimple2}
X_{ini}^j(t,\,n)=\begin{cases}
\ds-s_j\,n_{j+1}+d_j+(t-t_0)\,s^*\,k-s^*\,(n-n_{j+1}),\\
0\leq n-n_{j+1}\leq k(t-t_0);\\
\\
-s_j\,n+d_j+v_{max}(t-t_0),\\
n_j\leq n\leq n_{j+1}.
\end{cases}
\eeq

\item[2.] Upstream boundary value problem
\bel{upstreamsimple}
X_{up}^j(t,\,n)=\begin{cases}v^jt^{j+1}+b^j+s^*\big(k(t-t^{j+1})-(n-N_1)\big),\\
0\leq n-N_1\leq k(t-t^{j+1}); \\
\\
\ds v^j\,t+b^j-(n-N_1)\,{v^j\over k},\\
\max\{0,\,k(t-t^{j+1})\}\leq n-N_1\leq k(t-t^j) .
\end{cases}
\eeq

\item[3.] Downstream boundary value problem

\bel{downsimple}
X_{down}^j(t,\,x)=v_j\,t_{j+1}+b_j+(t-t_{j+1})\,v_{max},\quad (t,\,x)\in[t_{j+1},\,+\infty)\times \{N_2\}
\eeq

\item[4.] Internal boundary value problem
\bel{intsimple}
X_{int}(t,\,n)=\begin{cases}
\beta+\alpha\Big(\ds{n-n_{min}-kt+r\,t_{min}\over r-k}-t_{min}\Big)\\
n-n_{min}\geq r\,(t-t_{min}),\\
k(t-t_{max}) < n-n_{max}, \\
n-n_{min}\leq k(t-t_{min});\\
\\
\ds \beta+{n-n_{min}\over r}\,\alpha+s^*k\big(t-t_{min}-{n-n_{min}\over r}\big)\\
0 \leq n-n_{min} < r(t-t_{min})\quad\hbox{and}~~~n\leq n_{max};\\
\\
\ds\beta+\alpha\,(t_{max}-t_{min})+(t-t_{max})\,s^*k-s^*\,(n-n_{max}),\\
0 \leq n-n_{max}\leq k(t-t_{max});
\end{cases}
\eeq

\end{itemize}
\end{proposition}

\begin{remark}
 (\ref{intsimple}) is well defined even if $k=r$ and $r=0$. 
\end{remark}

The proof of Proposition \ref{explicit2} is straightforward. A simple verification reveals that the formulae for upstream, downstream and internal boundary conditions (\ref{upstreamsimple})-(\ref{intsimple}) can be merged into one, namely (\ref{intsimple}). In other words, the upstream and downstream conditions can be treated as spacial cases of internal boundary conditions. This is because the wave propagation speed is always non-negative, and any value condition can only affect the solution at region with a larger $n$. This coincides with the intuition that a car cannot be affected by others behind it. 

In order to compute the PDE solution with more complicated value conditions, for example, the one including multiple data,  we invoke the inf-morphism property from Proposition \ref{infmorphism}, and take the point wise minimum over all solutions obtained from (\ref{inisimple1}), (\ref{inisimple2}) and (\ref{intsimple}).

The Lax-Hopf formula is a grid-free numerical scheme in the sense that the computational procedure  does not rely on a two-dimensional grid, as opposed to finite-difference schemes. In addition, the solution algorithm is highly parallelizable: the problem of fusing multiple value conditions into one PDE can be decomposed into independent sub-problems. In each sub-problem, the solution can be solved efficiently with formulae (\ref{inisimple1})-(\ref{intsimple}).

\section{Numerical Study}\label{numerical}
In this section, we conduct a numerical case study using traffic data obtained from the Mobile Century field experiment \citep{MC}. The dataset includes vehicle trajectory data and cumulative passing vehicles measured  at different locations along a highway segment. We utilize a subset of the Mobile Century dataset and fuse them into the Lagrangian Hamilton-Jacobi equation. The goal is to perform the estimation of Lagrangian quantities such as vehicle trajectories and velocities.

\subsection{The Mobile Century field experiment}
On February 8, 2008, an experiment in the area of traffic monitoring was launched between 9:30 am to 6:30 pm on freeway I-880 near Union City in the San Francisco Bay Area, California. This experiment involved 100 vehicles  carrying GPS-enabled Nokia N95 phones. The probe vehicles repeatedly drove loops of 6 to 10 miles in length continuously for 8 hours.

Carried by each probe vehicle, each on-board smart phone stored its position and velocity every 3 to 4 seconds. In addition to the cell phone  data, the experiment also collected data on cumulative passing vehicles at different locations along the highway. The reader is referred to \cite{MC} for more details on experimental design and data description.

\subsection{Details of  the numerical study}
The freeway segment of our study is a 3.45 mile stretch of I-880 North Bound, between postmile 23.36, and  postmile 26.82. We utilize two types of data: (i) the cumulative vehicle count, obtained via loop detector station 400536 (postmile 23.36), which counts the number of passing vehicles every 30 seconds; (ii) vehicle trajectories, recorded by on-board smart phones every 3 to 4 seconds.  Our study spanned one hour,  from 11:30 am to 12:30 pm, and involved a total traffic volume of approximately 5000 vehicles. We utilized data sent from 97 probe vehicles, The vehicle trajectories are plotted in Figure \ref{mobiledata}. In order to identify the `label' of those probe vehicles, we processed the data on the cumulative vehicle count obtained from Station 400536.  

We assume that the fundamental diagram $\psi(\cdot)$ in (\ref{eqn2}) is piecewise affine. Such a fundamental diagram is calibrated with data collected form the same experiment (detail of the model calibration is omitted for brevity). The key parameters are estimated as follows.
$$
v_{max}~=~31.5~(meter/second),\quad \rho_{max}~=~0.50~(vehicle/meter),\quad \rho^*~=~0.055~(vehicle/meter)
$$
The resulting fundamental diagram reads
\bel{svestimate}
\psi(s)~=~\begin{cases}1.95\,s\qquad &\hbox{if}\quad 2.00~<~s~<~18.15\\
\\
31.5\qquad&\hbox{if}\quad s~\geq~18.15
\end{cases}
\eeq

\subsection{Numerical results}

We reconstructed the traffic state between postmile 23.36 and 26.82 for a period of one hour, using the numerical method discussed in Section  \ref{simplesol}. Among the 97 vehicle trajectories, we used 44 of them as training data. In other words, the viability episolution to equation (\ref{eqn2}) was computed with the Lax-Hopf formula and 44 internal boundary conditions. Notice that the penetration rate of probe vehicles is around $0.88\%$, given the total traffic volume in this one-hour period. Recall that  the on-board smart phone recorded vehicle location every 3 to 4 seconds. However, in our actual computation, we sampled the vehicle location at a much lower frequency, namely at every 30 to 90 seconds. As we show later in this section, the numerical performance of the model remains robust under such crude measurements.

\begin{figure}[h!]
\centering
\includegraphics[width=0.7\textwidth]{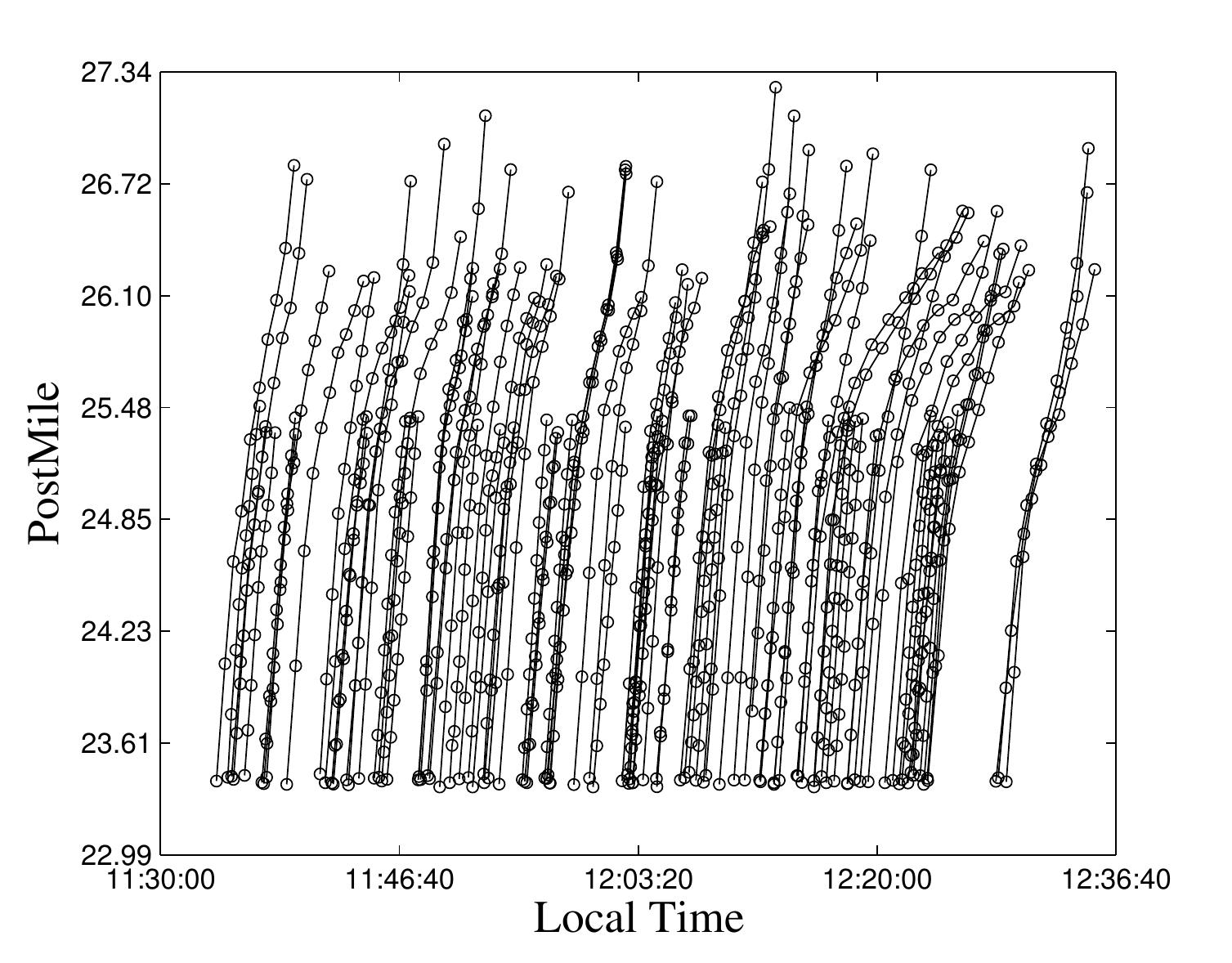}
\caption{Recorded trajectories of 97 probe vehicles.}
\label{mobiledata}
\end{figure}

Figure \ref{sol7800} displays the solution of the H-J equation (\ref{HJeqn}) based on 8 vehicle trajectories (internal boundary conditions). It should be noted that the viability episolution satisfies the value conditions only in an inequality sense \citep{ABSP}, that is 
$$
X_{\mathcal{C}}(t,\,n)~\leq~\mathcal{C}(t,\,n),\qquad (t,\,n)\in \hbox{Dom}(\mathcal{C})
$$
If the strict inequality holds, the value conditions and the model are said to be incompatible. The incompatibility is due to either error in measurements or inaccuracy of the model itself. A class of problems derived from incompatibility, namely data assimilation and data reconciliation, were discussed in \cite{Convex}.

\begin{figure}[h!]
\centering
\includegraphics[width=0.85\textwidth]{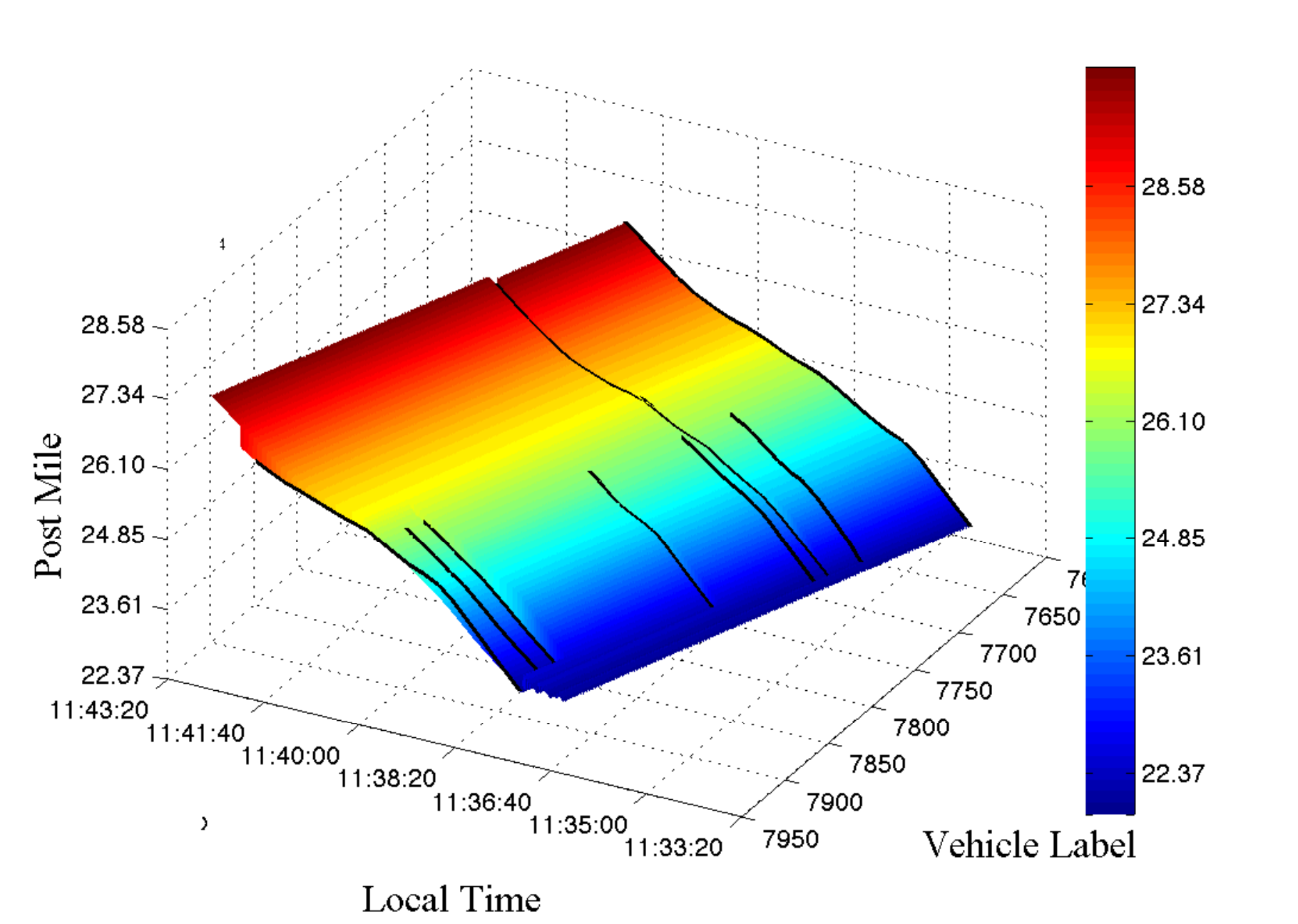}

\caption{Solution to the Hamilton-Jacobi equation (\ref{HJeqn}) with 8  (internal) boundary conditions. The solid lines are plots of vehicle trajectories.}
\label{sol7800}
\end{figure}

The obtained solution $X(t,\,n)$ describes the time evolution of the location of vehicle labeled $n$. In other words, it can be used to reconstruct vehicle trajectories. We fixed several values of $n$, and plotted the curves of $X(\cdot,\,n)$. Two examples are shown in Figure \ref{trajectory1} and \ref{trajectory2}. For comparison purpose, we also included the actual vehicle trajectory reported by the smart phone in the same figure. 
\begin{figure}[h!]
\begin{minipage}[b]{.49\textwidth}
\centering
\includegraphics[width=1.1\textwidth]{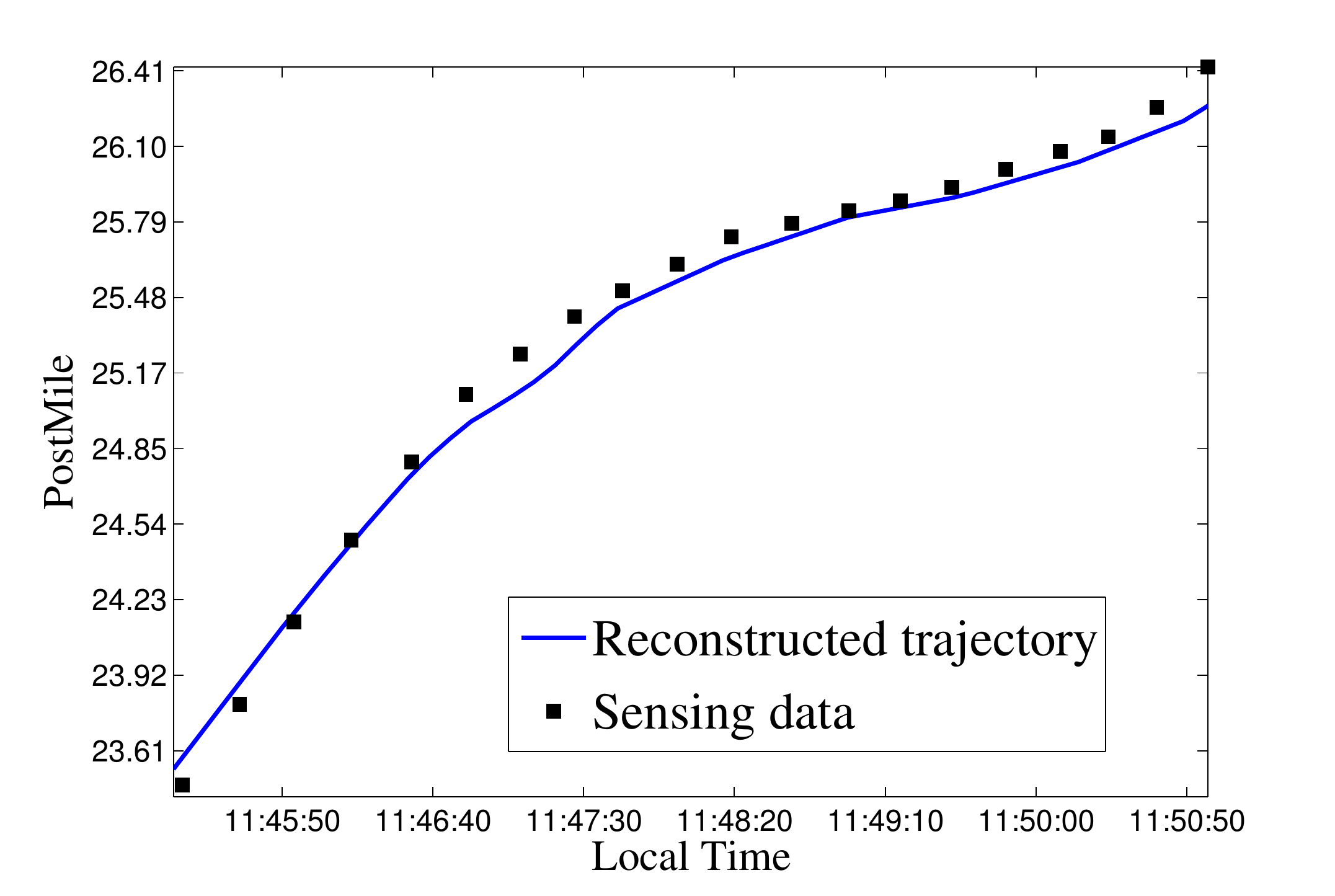}
\caption{\small Reconstructed and measured trajectories of vehicle \#8685}
\label{trajectory1}
\end{minipage}
\hspace{0.001cm}
\begin{minipage}[b]{.49\textwidth}
\centering
\includegraphics[width=1.1\textwidth]{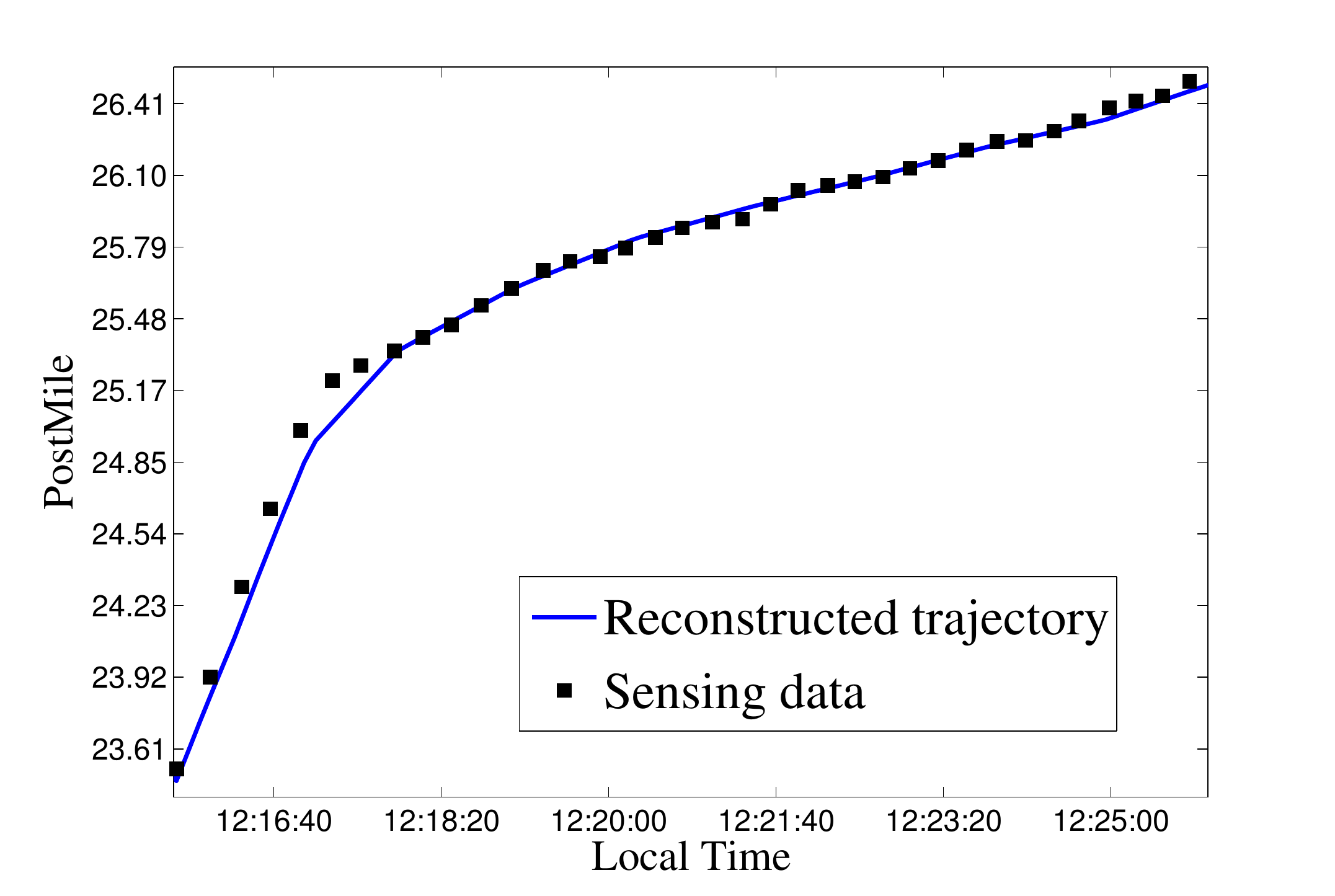}
\caption{\small Reconstructed and measured trajectories of vehicle \#11266}
\label{trajectory2}
\end{minipage}
\end{figure}

Recall that 
\bel{partialt}
\partial_t X(t,\,n)~=~\psi\Big(s(t,\,n)\Big)~=~v(t,\,n)
\eeq
By numerically differentiating $X(t,\,n)$ with respect to $t$, one obtains the velocity field $v(t,\,n)$. 
Figure \ref{vel} shows the velocity field involving vehicle labeled 7650 to 7950. This result was obtained by differentiating the Moskowitz function displayed in Figure \ref{sol7800}.  From Figure \ref{vel}, we clearly observe a time-varying congestion level experienced by the drivers: vehicles experience first free flow traffic, then congested traffic, and finally less congested traffic. This observation is confirmed in Figure \ref{mobiledata}, where a similar driving condition is shown for all probe vehicles.

Finally, we performed the task of travel time estimation using the Lagrangian approach. The travel times were estimated from the vehicle trajectories, which were again obtained by the viability episolution $X(t,\,n)$. In order to examine the accuracy of our estimation, we compared the actual travel times of the 97 probe vehicles with the aforementioned estimated travel times. The result is summarized in Figure \ref{error}.

\begin{figure}[h!]
\centering
\includegraphics[width=0.85\textwidth]{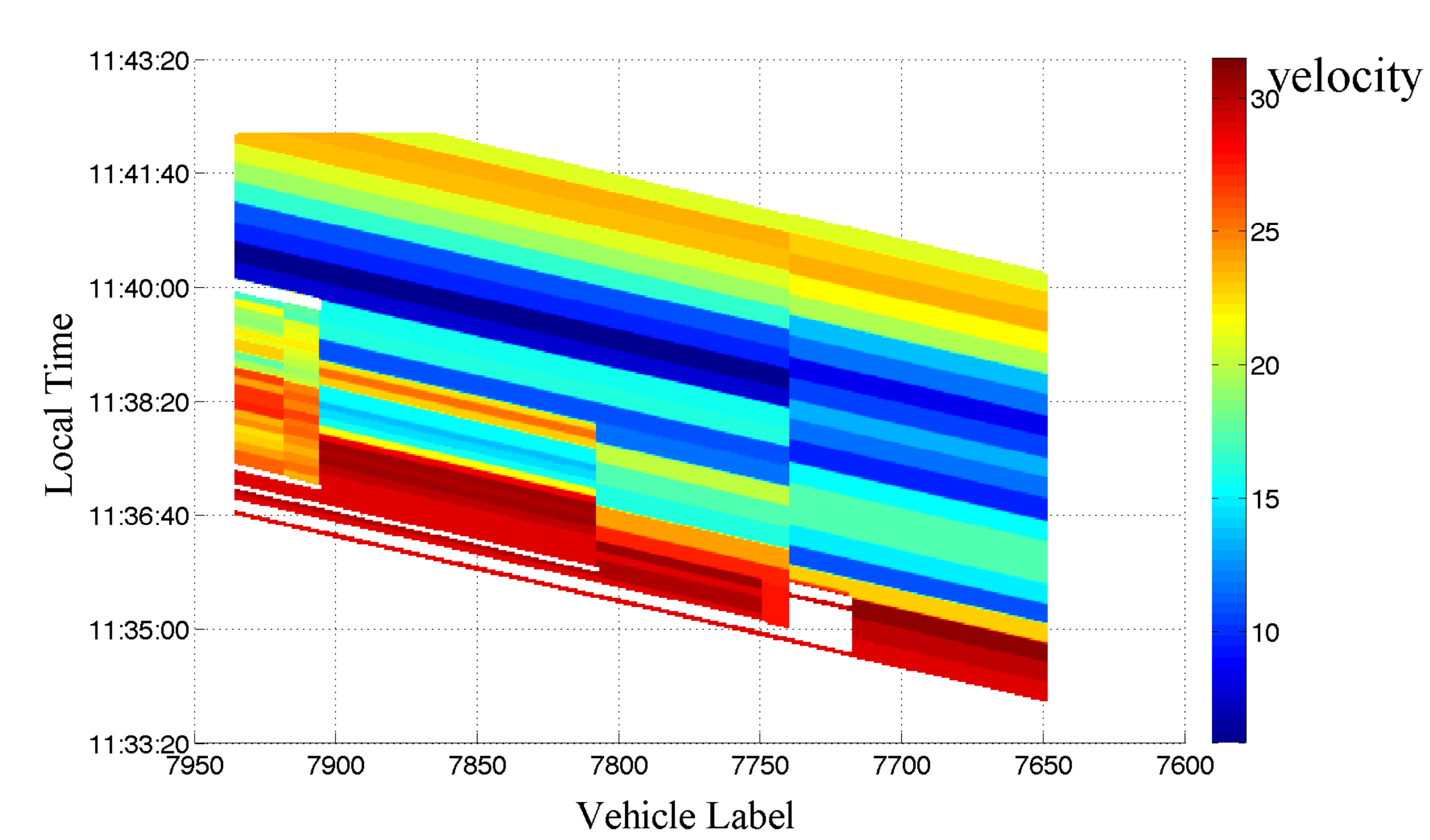}
\caption{Vehicle velocity field (in meter/second)  obtained from (\ref{partialt})}
\label{vel}
\end{figure}

\begin{figure}[h!]
\centering
\includegraphics[width=0.7\textwidth]{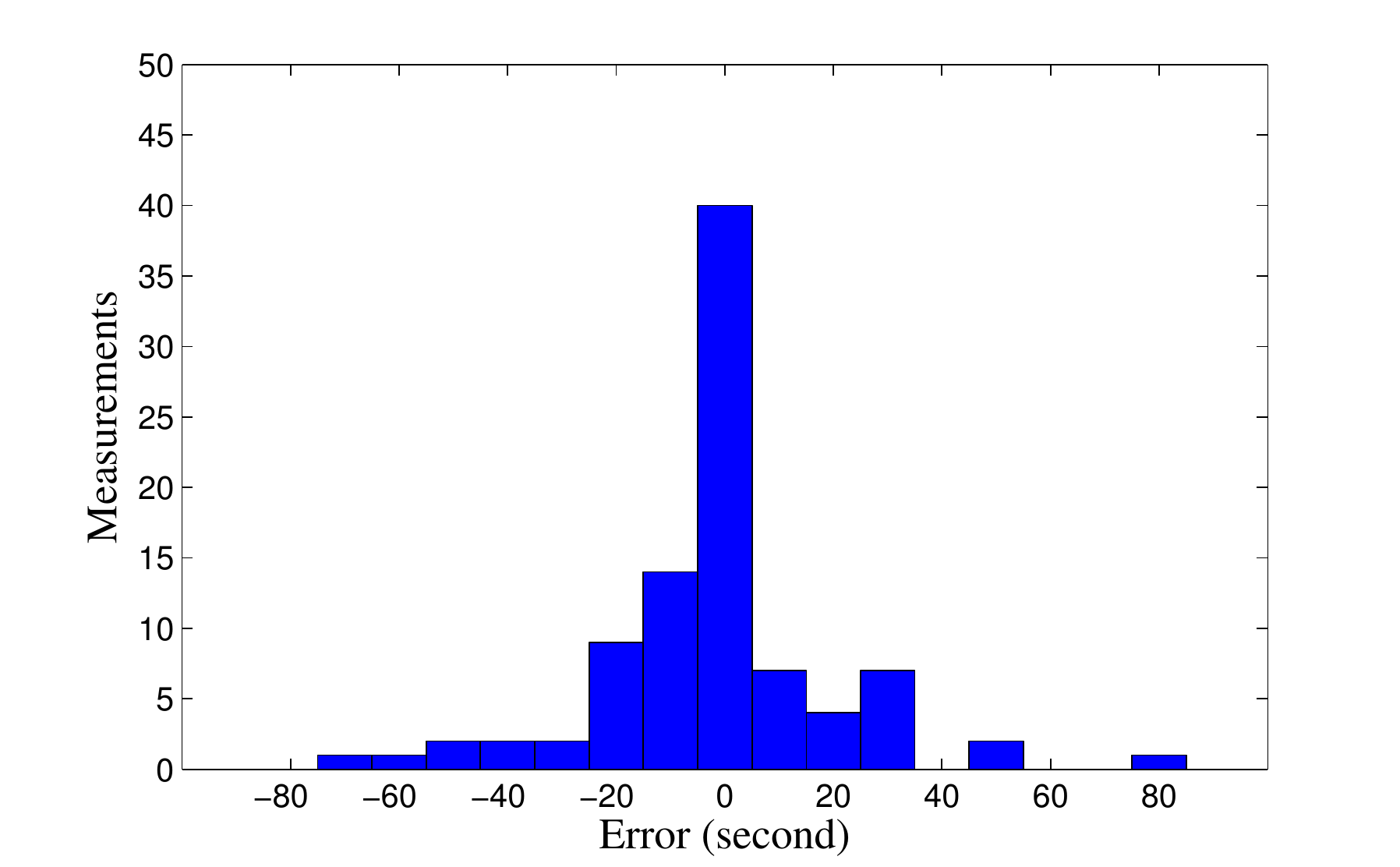}

\caption{Error of travel time estimation}
\label{error}
\end{figure}

\section{Conclusion}
This paper applies the Lagrangian-based first-order hydrodynamic traffic model to perform highway traffic estimation, utilizing both Eulerian and Lagrangian sensing data. We briefly discuss the Lighithill-Whitham-Richards model in both Eulerian and Lagrangian coordinate systems, and provide rigorous mathematical results on the equivalence of Hamilton-Jacobi equations in these two coordinate systems. We derive closed-form solutions to the LC-based Hamilton-Jacobi equation, and discuss the issue of fusing traffic data of different sources into the H-J equation via the notion of viability episolution. A case study of the Mobile Century project demonstrates the unique modeling features and insights offered by the Lagrangian-based PDE.

Despite several similarities between the EC-based and LC-based approaches such as structure of the PDE and solution method, the LC has the advantage of providing information associated with a given vehicle. Such a feature needs to be further explored and engineered to facilitate the development of cyber-physical infrastructure, including the {\it mobile internet}. In addition,  extension of LC-based PDE to a network setting is also an important aspect of future research.

\end{document}